\documentclass[12pt]{article}
\usepackage{amsmath,amsthm,amsfonts,latexsym,amscd,amssymb}
\usepackage{graphicx}
\scrollmode

\swapnumbers
\theoremstyle{plain}
\newtheorem{theorem}{Theorem}[section]
\newtheorem{lemma}[theorem]{Lemma}
\newtheorem{corollary}[theorem]{Corollary}

\theoremstyle{definition}

\theoremstyle{remark}

\newcommand{\forget}[1]{}

{\catcode`@=11\global\let\c@equation=\c@theorem}

\begin{document}
	\pagestyle{empty}

	\title{\textbf{ Some properties of the Zero-set Intersection graph of $C(X)$ and its Line graph} }
	\author{
		Yangersenba T Jamir$^{a}$ \& S Dutta$^{b}$  \\
		\footnotesize{$^{a, b}$Department of Mathematics}\\
		\footnotesize{$^{a, b}$North Eastern Hill University} {Shillong-22, Meghalaya, India}\\
		\footnotesize{$^{a}$yangersenbajamir@gmail.com, $^{b}$sanghita22@gmail.com}\\ 	
	}
\maketitle	
	\begin{abstract}
		Let $C(X)$ be the ring of all continuous real valued functions defined on a completely regular Hausdorff topological space $X$. The zero-set intersection graph $\Gamma (C(X))$ of $C(X)$ is a simple graph with vertex set all non units of $C(X)$ and  two vertices are adjacent if the intersection of the zero sets of the functions is non empty. In this paper, we study  the  zero-set intersection graph  of $C(X)$ and its line graph. We show that if $X$ has more than two points, then these graphs are connected with diameter and radius 2. We show that the  girth of the graph is  3 and the graphs  are both triangulated and hypertriangulated. We find the domination number of these graphs and finally we  prove that $C(X)$ is a von Neuman regular ring if and only if $C(X)$ is an almost regular ring and for all $f \in V(\Gamma (C(X)))$ there exists $g \in V(\Gamma(C(X)))$ such that $Z(f) \cap Z(g) = \phi$ and $\{f,g\}$ dominates $\Gamma(C(X))$. Finally, we derive some properties of the line graph of $\Gamma (C(X))$.
	
	\end{abstract}
	\vspace{.5 cm}
	\textbf{Keywords}: Zero-set intersection graph, diameter, girth, cycles, dominating sets, line graph
	\vspace{.5 cm}\\
	\textbf{Mathematics Subject Classification(2010)}: 54C40, 05C69	 
	\maketitle
	\section{Introduction}
	The characterization of various algebraic structures by means of graph theory is an interesting area of research for mathematicians in the recent years. It was started by Beck in $\cite{Beck}$ by defining the graph $\Gamma(R) $ of a commutative ring $R$  with vertices as elements of $R$ in which two different vertices $a$ and $b$ are adjacent if and only if $ab = 0$. The author studied finitely colorable rings by associating this graph structure and  this study was further continued by Anderson and Naseer in $\cite{Anderson}$. Sharma and Bhatwadekar first defined the comaximal graph of a commutative ring in $\cite{Sharma}$ and further investigation was continued in $\cite{Barati}$, $\cite{Afk}$,  $\cite{Habibi}$, $\cite{Dheena}$, $\cite{Jinnah}$, $\cite{Salimi}$, $\cite{Maimani}$, $\cite{Mehdi}$, $\cite{Moco}$,  $\cite{Petro}$, $\cite{Wang}$, $\cite{Ye}$. The comaximal ideal graph of $C(X)$, $\Gamma_2'C(X)$,  was studied in $\cite{Amini}$ and $\cite{Badie}$ where they derived the ring properties of $C(X)$  and  topological properties of $X$ via $\Gamma_2'(C(X))$. 
	
	Bose and Das in $\cite{Bose}$ introduced a graph structure called the  zero-set intersection graph $\Gamma(C(X))$ on the ring of real valued continuous functions $C(X)$ as a graph whose set of vertices consists of all non-units in the ring $C(X)$ and there is an edge between distinct vertices $f$ and $g$ in $C(X)$ if $Z(f) \cap Z(g) \neq \phi$. The authors showed that the graph is connected, studied the  cliques and maximal cliques of $\Gamma(C(X))$ and the inter-relationship of cliques of $\Gamma(C(X))$ and ideals in $C(X)$. Further they showed that two graphs are isomorphic if and only if the corresponding rings are isomorphic if and only if the corresponding topologies are homeomorphic either for first countable topological spaces or for realcompact topological spaces.
	
	Let $R$ be a commutative ring with unity. The ring $R$ is called an almost regular ring if each non-unit element of $R$ is a zero-divisor element of $R$. The comaximal graph ${\Gamma}(R)$ is defined as a graph with set of vertices as $R$ and two vertices $a,b$  are adjacent if and only if $Ra + Rb = R$.  Let ${\Gamma}_2(R)$ denote the subgraph  of ${\Gamma}(R)$  whose vertex set consists of all non-unit elements of $R$. If $J(R)$ is the Jacobson radical of $R$, then the graph with vertex set $V({\Gamma}_2(R)) \setminus J(R)$ is denoted by  ${\Gamma}_2'(R)$. For any two vertices $f$ and $g$ in ${\Gamma}_2'(C(X))$, $f$ is adjacent to $g$ if and only if $Z(f) \cap Z(g) = \phi$ from lemma 1.2 $\cite{Badie}$. The complement graph $\overline{{\Gamma}_2'C(X)}$ coincides with the zero-set intersection graph of $C(X)$. In this paper, we study the zero-set intersection graph of $C(X)$. Vertex set of $\Gamma(C(X))$ is denoted by $V({\Gamma}(C(X)))$ which consists of all  the non-units in $C(X)$ and two vertices $f$ and $g$ are  adjacent if $Z(f) \cap Z(g) \neq \phi$. If $X$ is singleton, then   ${\Gamma}(C(X))$ is empty. Thus we assume that $|X| > 1$.
	
	A graph is a pair G = (V, E), where V is a set whose elements are called vertices, and E is a set of pair of vertices, whose elements are called edges. Two elements $u$ and $v$ in $V$ are
	said to be adjacent if $\{u, v\} \in E$. A graph $G$ is said to be complete if every pair of vertices can be joined by an edge and $G$ is said to be connected if for any pair of vertices $u$, $v \in V$, there exists a path joining $u$ and $v$. The distance between two vertices $u$ and $v$ denoted by $d(u,v)$ is the length of the shortest path between them. The  diameter  is defined as diam$(G)$ = sup $\{d(u,v): u, v \in V(G)\}$, where $d(u, v)$ is the length of a shortest path from $u$ to $v$. The eccentricity of a vertex $u \in G$ is denoted by $ecc(u)$ and is defined as max\{$d(u,v) : v \in G$\}. The $min\{ecc(u) : u \in G\}$ is called the radius of $G$ and it is denoted by $Rad(G)$. The length of the shortest cycle in a graph $G$ is called  girth of the graph and is denoted by gr$(G)$. In a graph $G$, a dominating set is a set of vertices $A$ such that every vertex outside $A$ is adjacent to at least one vertex in $A$. A graph is said to be triangulated(hypertriangulated) if every vertex of the graph is a vertex(edge) of a triangle. If two distinct vertices $u$ and $v$ in a graph $G$ are adjacent and there is no vertex $w \in G$ such that $w$ is adjacent to both $u$ and $v$, then we say that $u$ and $v$ are orthogonal and is denoted by $u \perp v$. A graph $G$ is called complemented if for each vertex $u \in G$, there is a vertex $v \in G$ such that $u \perp v$. An edge which joins two vertices of a cycle but is not itself an edge of the cycle is a chord of that cycle. A graph is said to be chordal if every cycle of length greater than three has a chord. The line graph of $G$, denoted by $L(G)$, is a graph whose vertices are the edges of $G$ and two vertices of $L(G)$ are said to be adjacent wherever the corresponding edges of $G$ are incident on a common vertex $\cite{Lee}$. For any undefined term in graph theory, we refer the reader to $\cite{Wilson}$.

	In section 2 of this paper, we show that ${\Gamma}(C(X))$ is connected and find the diameter and radius of the graph. In section 3, we find the girth of the graph and show that the graph is always triangulated and hypertriangulated. We also find the conditions when the graph is chordal and show that ${\Gamma}(C(X))$ is never complemented. In section 4, we find the dominating number of ${\Gamma}(C(X))$ and show that $C(X)$ is a von Neuman regular ring if and only if $C(X)$ is an almost regular ring and for all $f \in V({\Gamma}(C(X)))$ there exists $g \in V({\Gamma}(C(X)))$ such that $Z(f) \cap Z(g) = \phi$ and $\{f,g\}$ dominates $\Gamma(C(X))$. Finally in  section 5, we study the line graph $  L(\Gamma(C(X)))$ of ${\Gamma}(C(X))$ and derive similar properties as in ${\Gamma}(C(X))$. For all notations and undefined terms concerning the ring $C(X)$, the reader may consult $\cite{Gillman}$ .\\

	\section{Diameter, radius of ${\Gamma}(C(X))$  }
	We first note that $f \in V({\Gamma}(C(X)))$ if and only if $Z(f) \neq \phi$. We study the condition when   ${\Gamma}(C(X))$ is connected. We also calculate the diameter and radius of  ${\Gamma}(C(X))$.
	
	\begin{theorem}
		Let $X$ be any topological space then  ${\Gamma}(C(X))$ is connected with $diam({\Gamma}(C(X)))$ = 2 if and only if $|X| > 2$.
	\end{theorem}
	
	\begin{proof}
		If $|X| = 1$, then $C(X) \cong \mathbb{R}$ and so ${\Gamma}(C(X))$ is empty. Suppose $X = \{a,b\}$, then $f$ is a vertex in ${\Gamma}(C(X))$ if $f(a) = 0$ and $f(b) \neq 0$, $f(a) \neq 0$ and $f(b) = 0$ or $f(a) = 0$ and $f(b) = 0$. Thus ${\Gamma}(C(X))$ is the disjoint union three complete subgraphs $A = \{f \in C(X) : f(a)= 0 ~and~ f(b) \neq 0\}$, $B = \{f \in C(X) : f(a)\neq 0 ~and~ f(b) = 0\}$ and $C = \{f \in C(X) : f(a)= 0 ~and~ f(b) = 0\}$. Thus $diam({\Gamma}(C(X)))$ = $\infty$. We now assume that $|X| > 2$. Let $f$, $g$ $\in$ $V({\Gamma}(C(X)))$ such that $Z(f) \cap Z(g) = \phi$. Let $a \in Z(f)$, $b \in Z(g)$ and $c \notin \{a,b\}$. By regularity of $X$, there exists an open set $U$ such that $c \in U \subseteq \overline{U} \subseteq X \setminus \{a,b\}$. Since $X$ is completely regular there exists $h_1, h_2 \in C(X)$ such that $h_1(\overline{U}) = 1, h_1(a) = 0, h_2(\overline{U}) = 1$ and $h_2(b) =0$. Let $h = h_1h_2$, then $h \in  V({\Gamma}(C(X))) \setminus \{f,g\}$ and $f-h-g$ is a path of length $2$ joining $f$ and $g$. Thus ${\Gamma}(C(X))$ is connected with diameter $2$.
	\end{proof}	
	
	\begin{corollary}
		If $|X| > 2$, then for any two distinct vertices $f$, $g$ $ \in  V({\Gamma}(C(X)))$ there exists a vertex $h \in  V({\Gamma}(CX)))$ such that $h$ is adjacent to both $f$ and $g$.
	\end{corollary}
	
	\begin{proof}
		If $Z(f) \cap Z(g) = \phi $, then by theorem 2.1 there exists $h \in V({\Gamma}(C(X)))$ such that $f-h-g$ is a path in ${\Gamma}(C(X))$. If $Z(f) \cap Z(g) \neq \phi $ then let  $h = 2f,  ~if~ g \neq 2f ~and~ h = 3f$ otherwise. Then $h \in V({\Gamma}(C(X))) \setminus \{f,g\}$ and $f-h-g$ is a path in  ${\Gamma}(C(X))$.	
		
	\end{proof}

	\begin{corollary}
		Let $f,g \in {\Gamma}(C(X))$. Then\\ 
		$~(1)~ d(f,g) = 1 ~if~ and ~only~ if ~Z(f) \cap Z(g) \neq \phi $.\\ 
		$~(2)~  d(f,g) = 2~ if ~and~only ~ if ~Z(f) \cap Z(g) = \phi $.
	\end{corollary}

	\begin{theorem}
		For any space $X$ with $|X| > 2$, $Rad({\Gamma}(C(X))) = 2$.
	\end{theorem}	
	\begin{proof}
		For any vertex $f \in {\Gamma}(C(X))$, it is evident that $1 \leq e(f) \leq 2$, so $1 \leq Rad({\Gamma}(C(X))) \leq 2$. Let $f \in C(X) \setminus \{0\}$ and $x_1, x_2, x_3$ be three distinct points in $X$. Since $X$ is completely regular Hausdorff space there exist mutually disjoint open sets $U_1, U_2$ and $U_3$ such that $x_i \in U_i$ for each $i \in \{1,2,3\}$. Let $g = 1 - f$ $\in C(X) \setminus \{0\}$ be such that $g(x_i) = 0$ and $g(X \setminus U_i) = 1 ~\forall~ i \in \{1,2,3\}$. Then $Z(f) \cap Z(g) = \phi$. Thus, for every $f \in   V({\Gamma}(C(X)))$ there exists $g \in  V({\Gamma}(C(X)))$ such that $Z(f) \cap Z(g) = \phi$ and so $2 = d(f,g) \leq e(f) \leq 2$. Therefore, $Rad({\Gamma}(C(X))) = 2$.
	\end{proof}		
	\section{Cycles in ${\Gamma}(C(X))$ }
	In this secton we explore the existence of cycles in  ${\Gamma}(C(X))$.
	\begin{theorem}
		For any topological  space $X$ with $|X| > 1$, ${\Gamma}(C(X))$ is both triangulated and hypertriangulated.
	\end{theorem}
	
	\begin{proof}
		Let $f \in  V({\Gamma}(CX)))$. Then $f-2f-3f-f$ is a cycle of length 3 in ${\Gamma}(C(X))$. Hence, ${\Gamma}(C(X))$ is triangulated. Suppose $f$ and $g$  be adjacent vertices in ${\Gamma}(C(X))$. Let $h = 2f, ~if~ g \neq 2f$,  otherwise let $h = 3f$. Then $h \in  V({\Gamma}(C(X))) \setminus \{f,g\}$ and $f-g-h-f$ is a triangle in ${\Gamma}(C(X))$. Hence, $f-g$ is an edge in a triangle. Thus, ${\Gamma}(C(X))$ is hypertriangulated.
	\end{proof}
	
	\begin{corollary}
		If $|X| > 1$, then $gr({\Gamma}(C(X))) = 3$.
	\end{corollary}
	
	\begin{theorem}
		Let $f~and~g$ be two distinct vertices in ${\Gamma}(C(X))$. Then \\
		(1) $c(f,g) = 3~if~and~only~if~ Z(f) \cap Z(g) \neq \phi $.\\
		(2) $c(f,g) = 4~if~and~only~if~ Z(f) \cap Z(g) = \phi $.
	\end{theorem}
	
	\begin{proof}
		(1)~Let $c(f,g) = 3$. Then $f$ and $g$ are adjacent in ${\Gamma} (C(X))$. Thus  $Z(f) \cap Z(g) \neq \phi $. 
		
		Conversely, if  $Z(f) \cap Z(g) \neq \phi $ then $f$ and $g$ are adjacent in ${\Gamma}(C(X))$ and by corollary 2.2 there exists a vertex $h \in V({\Gamma}(C(X)))$ such that $h$ is adjacent to both $f$ and $g$. Thus $f-g-h-f$ is a cycle containing $f$ and $g$. Hence, $c(f,g) = 3$.\\
		
		(2) Let $c(f,g) = 4$. Then  $Z(f) \cap Z(g) = \phi $ by (1).
		
		 Conversely, suppose  $Z(f) \cap Z(g) = \phi $. Then by (1) there is no cycle of length 3 containing $f$ and $g$. By corollary 2.2, there exists $h \in V({\Gamma}(C(X)))$ such that $h$ is adjacent to both $f$ and $g$. So, $f-h-g-2h-f$ is a cycle of length 4 and it is the smallest cycle containing $f$ and $g$. Hence, $c(f,g) = 4$.
	\end{proof}
	
	\begin{theorem}
		${\Gamma}(C(X))$ is chordal if and only if $|X| \leq 3$.
	\end{theorem}
	
	\begin{proof}
		Suppose ${\Gamma}(C(X))$ is chordal and let $|X| \geq 4$. Let $x_1,x_2,x_3,x_4 \in X$. Since $X$ is a Tychonoff space,  there exist mutually disjoint open sets $U_1,U_2,U_3~and~U_4$ such that $x_i \in U_i$, where $i \in \{1,2,3,4\}$. Let $h_i \in V({\Gamma}(C(X)))$ such that $h_i(x_i) = 0$, and $h_i(X \setminus U_i) = 1$ for each $i \in \{1,2,3,4\}$. Then $Z(h_i) \cap Z(h_j) = \phi$, whenever $i \neq j$. Consider the functions $f_1 = h_1h_4,~ f_2 = h_1h_2,~ f_3 = h_2h_3$ and $f_4 = h_3h_4$. Clearly, $f_i \in  V({\Gamma}(C(X))) ~ (i = 1,2,3,4)$ and $f_1-f_2-f_3-f_4-f_1$ is a chordless cycle since $Z(h_1) \cap Z(h_3) = \phi$ and $Z(h_2) \cap Z(h_4) = \phi$ which is a contradiction. Hence, ${\Gamma}(C(X))$ is chordal only when $|X| \leq 3$. 
		
		Conversely, if $|X| \leq 3$ then the following two cases arise.\\
		Case I: If $X = \{a,b\}$, then $f$ is a vertex in ${\Gamma}(C(X))$  if $f(a) = 0$ and $f(b) \neq 0$, $f(a) \neq 0$ and $f(b) = 0$ or $f(a) = 0$ and $f(b) = 0$. Thus,  ${\Gamma}(C(X))$ is the disjoint union of three complete subgraphs $A = \{f \in C(X) : f(a)= 0 ~and~ f(b) \neq 0\}$, $B = \{f \in C(X) : f(a)\neq 0 ~and~ f(b) = 0\}$ and $C = \{f \in C(X) : f(a)= 0 ~and~ f(b) = 0\}$. Hence, if $C$ is an induced cycle in ${\Gamma}(C(X))$, then it is contained in the complete subgraph induced by $A, B$ or $C$ and so $C$ has a chord of length greater than $3$. Therefore,  ${\Gamma}(C(X))$ is chordal.\\
		Case II: Let $X = \{a,b,c\}$ and $C$ be an induced cycle in  ${\Gamma}(C(X))$ of length greater than 3. Consider a path $f_1-f_2-f_3-f_4$ in $C$ such that $f_1,f_2,f_3$ and $f_4$ are distinct. If $Z(f_1) \cap Z(f_3) \neq \phi$ or $Z(f_2) \cap Z(f_4) \neq \phi$, then we have a chord joining $f_1$ and $f_3$ or a chord joining $f_2$ and $f_4$. So we assume that  $Z(f_1) \cap Z(f_3) = \phi$ and $Z(f_2) \cap Z(f_4) = \phi$. Then we must have $|Z(f_2)| = |Z(f_3)| = 2$. If $f_1$ and $f_4$ are adjacent then, there is a chord joining $f_2$ and $f_4$. So assume that $f_1$ and $f_4$ are not adjacent then, there is a vertex $f_5$ which is adjacent to $f_4$. If $|Z(f_5)| = 1$, then $f_5$ is adjacent to $f_3$ and if  $|Z(f_5)| = 2$, then either $f_5$ is adjacent to $f_2$ and $f_3$ or $f_5$ is adjacent to $f_1$ and $f_2$. Thus in each case $C$ has a chord and hence, ${\Gamma}(C(X))$ is chordal.
	\end{proof}
	
	\begin{theorem}
		For $|X| > 1$, ${\Gamma}(C(X))$ is not complemented.
	\end{theorem}
	\begin{proof}
		Suppose ${\Gamma}(C(X))$ is complemented. Then for every $f \in \overline{{\Gamma}_2'C(X)}$ there exist $g$ such that $f \perp g$, thus $c(f,g) = 4$. By theorem 3.3, $Z(f) \cap Z(g) = \phi$, which is a contradiction.
	\end{proof}
	
	\section{Dominating Sets in ${\Gamma}(C(X))$}
	
	\begin{theorem}
		If $|X| > 1$, then $dt({\Gamma}(C(X))) = 2$.
	\end{theorem}
	
	\begin{proof}
		It is evident that $dt({\Gamma}(C(X))) \neq 1$. Let $f \in V({\Gamma}(C(X)))$ and $a \in Int_XZ(f)$, then $f \in O^a$. Since $O^a$ is a pure ideal, there exists $g \in O^a~ \subseteq Z(C(X))$ such that $f = fg$, that is $g = 1$ on $Supp(f)$ \cite{Ezeh}. Hence, $g^2 \in V({\Gamma}C(X))$ and $\{g,g^2\}$ dominates $V({\Gamma}C(X))$. Therefore, $dt({\Gamma}(C(X))) = 2$.
	\end{proof}
	
	\begin{theorem}
		If $|X| > 1$, then for any $f,g \in V({\Gamma}(C(X)))$, $\{f,g\}$ dominates ${\Gamma}(C(X))$ if and only if $Z(f) \cup Z(g) = X$.
	\end{theorem}
	
	\begin{proof}
		Suppose $\{f,g\}$ dominates ${\Gamma}(C(X))$ and assume that $y \in X \setminus (Z(f) \cup Z(g))$. Let $h \in C(X)$ such that $h(y) = 0$ and $h(Z(f) \cup Z(g)) = 1$. Then, $h \in  V({\Gamma}(C(X)))$, $h \neq \{f,g\}$, $Z(f) \cap Z(h) = \phi$ and $Z(g) \cap Z(h) = \phi$. This is a  contradiction since $\{f,g\}$ dominates ${\Gamma}(C(X))$.
		
		 Conversely, suppose $Z(f) \cup Z(g) = X$.  Let $h \in C(X) $ such that  $Z(f) \cap Z(h) = \phi$ and $Z(g) \cap Z(h) = \phi$.  Then, $Z(h) \cap (Z(f) \cup Z(g)) = \phi$ which implies that  $Z(h) = \phi$, i.e., $h \notin  V({\Gamma}(C(X)))$. Thus, $\{f,g\}$ dominates ${\Gamma}(C(X))$.
	\end{proof}

	\begin{theorem}
		$C(X)$ is a von Neuman regular ring if and only if $C(X)$ is an almost regular ring and for all $f \in V({\Gamma}(C(X)))$ there exists $g \in V({\Gamma}(C(X)))$ such that $Z(f) \cap Z(g) = \phi$ and $\{f,g\}$ dominates ${\Gamma}(C(X))$.
	\end{theorem}
	
	\begin{proof}
		Suppose $C(X)$ is a von Neuman regular ring.  Then it is clear that $C(X)$ is an almost regular ring and for all $f \in V({\Gamma}(C(X)))$, $Z(f)$ is open. Consider the function $g \in C(X)$ such that  
		\begin{equation*}
			g(x) =
			\begin{cases}
				1, ~~~~~~x \in Z(f)\\
				0, ~~~~~~x \notin Z(f).
				
			\end{cases}
		\end{equation*}
		Then $g \in V({\Gamma}(C(X)))$ and $Z(f) \cup Z(g)$ = $Z(f) \cup (X \setminus Z(f))$ = $X$. Hence, by theorem 4.2 $\{f,g\}$ dominates ${\Gamma}(C(X))$.
		
		 Conversely, suppose $C(X)$ is an almost regular ring and for all $f \in V({\Gamma}(C(X)))$ there exists $g \in V({\Gamma}(C(X)))$ such that  $Z(f) \cap Z(g) = \phi$ and $\{f,g\}$ dominates ${\Gamma}(C(X))$. Since $X$ is an almost $P$-space, every non-unit element in $C(X)$ has a zero-set with non-empty interior. That is every element in $C(X)$ is either a unit or a zero-divisor. If $f \in C(X)$ is a unit then, $f = f^2f^{-1}$. Suppose now $f \in  V({\Gamma}(C(X)))$. By hypothesis $Z(f) \cap Z(g) = \phi$, so $Coz(f) \cup Coz(g) = X$ which implies that $Z(f)$ is open since $Z(f) = Coz(g)$. Therefore, $C(X)$ is a von Neuman regular ring.
	\end{proof} 
	
	The following theorem is a direct consequence of theorem 5.3 $\cite{Badie}$ and theorem 4.2. 
	
	\begin{theorem}
		The graph ${\Gamma}(C(X))$ is complemented if and only if for every $f \in V({\Gamma}(C(X)))$ there exists $g \in V({\Gamma}(C(X)))$ such that $Z(f) \cap Z(g) = \phi$ and $\{f,g\}$ dominates ${\Gamma}(C(X))$.  
	\end{theorem}
	\section{Line graph of ${\Gamma}(C(X))$}
	Suppose $f,g \in  V({\Gamma}(C(X)))$, then $[f,g]$ is a vertex in $L({\Gamma}(C(X)))$ if $Z(f) \cap Z(g) \neq \phi$. In $L({\Gamma}(C(X)))$, $[f,g] = [g,f]$ as ${\Gamma}(C(X))$ is an undirected graph and for two distinct vertices $[f_1,f_2]$ and $[g_1,g_2]$ in $L({\Gamma}(C(X)))$, $[f_1,f_2]$ is adjacent to $[g_1,g_2]$ if $f_i = g_j$, for some $i,j \in \{1,2\}$.\\
	
	We first investigate when is $L({\Gamma}(C(X)))$ connected and then compute  diameter and radius of $L({\Gamma}(C(X)))$.
	\begin{lemma}
		Let $[f_1,f_2]$ and $[g_1,g_2]$ be distinct vertices in  $L({\Gamma}(C(X)))$. Then there is a vertex $[h_1,h_2]$ which is adjacent to both  $[f_1,f_2]$ and $[g_1,g_2]$ in $L({\Gamma}(C(X)))$ if and only if $Z(f_i) \cap Z(g_j) \neq \phi$ for some $i,j \in \{1,2\}$.
	\end{lemma}
	
	\begin{proof}
		Suppose there exists a vertex $[h_1,h_2]$ which is adjacent to both  $[f_1,f_2]$ and $[g_1,g_2]$ in $L({\Gamma}(C(X)))$. If $f_i = g_j$ for some $i,j \in \{1,2\}$, then $Z(f_i) \cap Z(g_j) = Z(f_i) \neq \phi$. So assume that $f_i \neq g_j$  for all $i,j \in \{1,2\}$. Consider a path $[f_1,f_2] - [h_1,h_2] - [g_1,g_2]$ in $L({\Gamma}(C(X)))$, then $h_1 = f_i$ for some $i \in \{1,2\}$ and $h_2 = g_j$ for some $j \in\{1,2\}$. Thus $Z(f_i) \cap Z(g_j) = Z(h_1) \cap Z(h_2) \neq \phi$.\\
		Conversely, suppose $Z(f_i) \cap Z(g_j) \neq \phi$ for some $i,j \in \{1,2\}$. Without loss of generality let $f_1 \neq g_1$, then $[f_1,g_1]$ is adjacent to both $[f_1,f_2]$ and $[g_1,g_2]$ in $L({\Gamma}(C(X)))$. If $f_1 = g_1$ then there exists $ r \in \mathbb{R} \setminus \{0,1\}$ such that $g_2 \neq rg_1$ and $f_2 \neq rf_1$. Thus $[f_1, rf_1]$ is adjacent to both $[f_1,f_2]$ and $[g_1,g_2]$.
		
	\end{proof}

	\begin{theorem}
		Let $|X| > 2$ and $[f_1,f_2], ~ [g_1,g_2]$ be distinct vertices in $L({\Gamma}(C(X)))$. Then\\
		(1) $d([f_1,f_2], [g_1,g_2]) = 1$ if and only if $f_i = g_j$ for some $i,j \in \{1,2\}$.\\
		(2) $d([f_1,f_2], [g_1,g_2]) = 2$ if and only if $f_i \neq g_j$ for all $i,j \in \{1,2\}$ and $Z(f_i) \cap Z(g_j) \neq \phi$ for some $i,j \in \{1,2\}$.\\
		(3) $d([f_1,f_2], [g_1,g_2]) = 3$ if and only if $f_i \neq g_j$ for all $i,j \in \{1,2\}$ and $Z(f_i) \cap Z(g_j) = \phi$ for all $i,j \in \{1,2\}$.
	\end{theorem}
	
	\begin{proof}
		(1) Clearly holds.\\
		(2) Suppose $d([f_1,f_2], [g_1,g_2]) = 2$. Then  $[f_1,f_2]-[f_i, g_j]- [g_1,g_2]$ is a path of length 2 for any $i,j \in \{1,2\}$ implies that $f_i \neq g_j$ for all $i,j \in \{1,2\}$ and $Z(f_i) \cap Z(g_j) \neq \phi$ for some $i,j \in \{1,2\}$.\\
		The converse follows clearly.  \\ 
		(3) From (1) and (2) we have $d([f_1,f_2], [g_1,g_2]) > 2$. By corollary 2.2, there exists $h \in  V({\Gamma}(C(X)))$ such that $h$ is adjacent to both $f_1$ and $g_1$ in  ${\Gamma}(C(X))$. Clearly, we can choose $h$ such that $ h \notin \{f_2,g_2\}$. Thus $[f_1,f_2] - [f_1,h] - [h,g_1] - [g_1,g_2]$ is  a path in $L({\Gamma}(C(X)))$ and hence $d([f_1,f_2], [g_1,g_2]) = 3$
	\end{proof}	
	
	The following corollary is immediate theorem 2.1 and theorem 5.2.
	
	\begin{corollary}
		If $|X| > 2$, then $L({\Gamma}(C(X)))$ is a connected graph with $diam(L({\Gamma}(C(X)))$ $\leq 3$.
	\end{corollary}
	
	
	\begin{theorem}
		Let $|X| > 2$ and $[f_1,f_2]$ be a vertex in $L({\Gamma}(C(X)))$. Then 
		\begin{equation*}
			e([f_1,f_2]) =
			\begin{cases}
				2, ~~~~~~if ~ Z(f_1) \cup Z(f_2) = X\\
				3, ~~~~~~otherwise.
				
			\end{cases}
		\end{equation*}
	\end{theorem}
	
	\begin{proof}
		It is clear that $1 \leq e([f_1,f_2]) \leq 3$. Let $[g_1,g_2]$ be a vertex in  $L({\Gamma}(C(X)))$ such that $[g_1,g_2]$ is not adjacent to $[f_1,f_2]$. Thus for all $i,j \in \{1,2\}$, $g_i \neq f_j$. But $\phi \neq Z(g_1) = Z(g_1) \cap X = Z(g_1) \cap (Z(f_1) \cup Z(f_2))$. Suppose,  $Z(f_1) \cap Z(g_1) \neq \phi$ then, $[f_1,f_2] - [f_1,g_1] - [g_1,g_2]$ is a path in $L({\Gamma}(C(X)))$ and so $d([f_1,f_2], [g_1,g_2]) = 2$. Hence,  $e([f_1,f_2]) =2$. Now, suppose $y \in X \setminus Z(f_1) \cup Z(f_2)$ and $V$ be an open set in $X$ such that $y \in V \subseteq Cl_X{V} \subseteq X \setminus Z(f_1) \cup Z(f_2)$. Consider $g_1, g_2 \in C(X)$, $g_1\neq g_2$ such that $g_i(y) = 0$ and $g_i(Z(f_1) \cup Z(f_2)) = 1$ for $i \in \{1,2\}$. Then $f_i \neq g_j$ and $Z(f_i) \cap Z(g_j) = \phi$ for all $i,j \in \{1,2\}$. So by theorem 2.1, $d([f_1,f_2], [g_1,g_2]) = 3$ and hence $e([f_1,f_2]) = 3$.
	\end{proof}	
	
	An immediate conclusion from corollary 2.2 and theorem 2.4 is the following corollary.
	
	\begin{corollary}
		If $|X| > 2$, then $2 \leq Rad(L({\Gamma}(C(X)))) \leq 3$.
	\end{corollary}
	
	We now find the girth of $L({\Gamma}(C(X)))$ and show that $L({\Gamma}(C(X)))$ is always triangulated and hypertriangulated. We also show that $L({\Gamma}(C(X)))$ is never chordal.
	
	\begin{theorem}
		If $|X| > 1$, then $gr(L({\Gamma}(C(X)))) = 3$.
	\end{theorem}
	
	\begin{proof}
		Let $[f,g]$ be a vertex in $L({\Gamma}(C(X)))$, then $[f,g] - [g,2f] - [2f,f] - [f,g]$ is a cycle of length $3$. Hence, $gr(L({\Gamma}(C(X)))) = 3$.
	\end{proof}
	
	\begin{theorem}
		For any space $X$ with $|X| > 1$,  $L({\Gamma}C(X)) $ is both triangulated and hypertriangulated .
	\end{theorem}
	
	\begin{proof}
		By theorem 5.6, it follows that $L({\Gamma}(C(X))) $ is triangulated. Let $[f_1,f_2] - [f_1,g]$ be an edge in $L({\Gamma}(C(X))) $. Then, $f_1 \neq 2f_1$ and $f_2 \neq 2f_1$ and $[f_1,f_2] - [f_1,g] - [f_1,2f_1] - [f_1,f_2]$ is a cycle in $L({\Gamma}(C(X))) $. Hence, $L({\Gamma}(C(X))) $ is hypertriangulated.
	\end{proof}

	\begin{theorem}
		If $[f_1,f_2]$ and $[g_1,g_2]$ are distinct vertices in $L({\Gamma}(C(X))) $. Then\\
		(1) $c([f_1,f_2],[g_1,g_2]) = 3$ if and only if $f_i = g_j$ for some $i,j \in \{1,2\}$.\\
		(2)  $c([f_1,f_2],[g_1,g_2]) = 4$ if and only if $f_i \neq g_j$ for all $i,j \in \{1,2\}$ and $Z(f_i) \cap Z(g_j) \neq \phi$ for some $i \in \{1,2\} ~and~ for~ all~ j \in \{1,2\}$ or ($Z(f_1) \cap Z(g_i) \neq \phi$ and $Z(f_2) \cap Z(g_j) \neq \phi$, where $i,j \in \{1,2\}$).\\
		(3) $c([f_1,f_2],[g_1,g_2]) = 5$ if and only if $f_i \neq g_j$ for all $i,j \in \{1,2\}$ and for only one $i \in \{1,2\}$ and only one $j \in \{1,2\}$, $Z(f_i) \cap Z(g_j) \neq \phi$.\\
		(4)  $c([f_1,f_2],[g_1,g_2]) = 6$ if and only if $f_i \neq g_j$ for all $i,j \in \{1,2\}$ and $Z(f_i) \cap Z(g_j) = \phi$ for all $i,j \in \{1,2\}$.
	\end{theorem}
	
	\begin{proof}
		(1). If $c([f_1,f_2],[g_1,g_2]) = 3$, then it is clear that $f_i = g_j$ for some $i,j \in \{1,2\}$.\\
		Conversely, suppose $f_1 = g_1$. Then there exists $r \in \mathbb{R} \setminus \{0\}$ such that $rg_1 \notin \{f_1,f_2,g_1\}$ and $[f_1,f_2] - [f_1,g_1] - [f_1,rg_1] - [f_1,f_2]$ is a cycle of length 3 in $L({\Gamma}(C(X))) $ containing $[f_1,f_2]$ and $[g_1,g_2]$. Hence,  $c([f_1,f_2],[g_1,g_2]) = 3$.\\
		
		(2). Suppose   $c([f_1,f_2],[g_1,g_2]) = 4$ then by (1), $f_i \neq g_j$ for all $i,j \in \{1,2\}$. So we have the cycle $[f_1,f_2] - [a,b] - [g_1,g_2] - [c,d] - [f_1,f_2]$, where $a,c \in \{f_1,f_2\}$ and $b,d \in \{g_1,g_2\}$. \\
		$\textbf{Case I}$ : If $a = f_1 = c$, $b = g_1$ and $ d = g_2$ then $Z(f_1) \cap Z(g_1) \neq \phi$ and $Z(f_1) \cap Z(g_2) \neq \phi$.\\
		$\textbf{Case II}$ : If $a = f_1$, $c = f_2$, $b = g_1$ and $d = g_2$, then $Z(f_1) \cap Z(g_1) \neq \phi$ and $Z(f_2) \cap Z(g_2) \neq \phi$. \\
			$\textbf{Case III}$ : If $a = f_1 = c$ and $b = g_2 = d$, then $[a,b] = [c,d]$, which is a contradiction.\\
		Conversely, if $f_i \neq g_j$ for all $i,j \in \{1,2\}$ then  there is no triangle containing $[f_1,f_2]$ and $[g_1,g_2]$. Suppose $Z(f_2) \cap Z(g_j) \neq \phi$ for all $j \in \{1,2\}$ then $[f_1,f_2] - [f_2,g_2] - [g_1,g_2] - [g_1,f_2] - [f_1,f_2]$ is a cycle of length $4$ in $L({\Gamma}(C(X))) $ containing  $[f_1,f_2]$ and $[g_1,g_2]$. Suppose $Z(f_1) \cap Z(g_2) \neq \phi$ and $Z(f_2) \cap Z(g_2) \neq \phi$ then we get a cycle $[f_1,f_2] - [f_1,g_2] - [g_1,g_2] - [g_2,f_2] - [f_1,f_2] $ of length $4$ in $L({\Gamma}(C(X))) $. Hence, $c([f_1,f_2],[g_1,g_2]) = 4$.\\
		
		(3). Suppose $c([f_1,f_2],[g_1,g_2]) = 5$. Then by (1), $f_i \neq g_j$ for all $i,j \in \{1,2\}$. Suppose $Z(f_i) \cap Z(g_j) = \phi$ for all $i,j \in \{1,2\}$ then, $[f_1,f_2] - [k_1,l_1] - [g_1,g_2] - [k_2,l_2]-[k_3,l_3] - [f_1,f_2]$ is a cycle of length $5$ in $L({\Gamma}(C(X))) $, where $k_1 \in \{f_1,f_2\}$ and $l_1 \in \{g_1,g_2\}$. But $Z(k_1) \cap Z(l_1) \neq \phi$ which is a contradiction to our assumption. Similarly, if we consider the cycle $[f_1,f_2] - [k_1,l_1] - [k_2,l_2] - [g_1,g_2] - [k_3,l_3] - [f_1,f_2]$, we get a contradiction. Hence, for only one $i \in \{1,2\}$ and only one $j \in \{1,2\}$, $Z(f_i) \cap Z(g_j) \neq \phi$.\\
		Conversely, suppose $f_i \neq g_j$ for all $i,j \in \{1,2\}$ and for only one $i \in \{1,2\}$ and only one $j \in \{1,2\}$, $Z(f_i) \cap Z(g_j) \neq \phi$. Let $Z(f_2) \cap Z(g_2) \neq \phi$. Then by (1) and (2) it follows that there is no cycle of length 3 or 4 containing $[f_1,f_2]$ and $[g_1,g_2]$. Now there exists $r \in \mathbb{R} \setminus \{0\}$ such that $rg_2 \notin \{f_1,g_1,g_2\}$ and so $[f_1,f_2] - [f_2,g_2] - [g_1,g_2] - [g_2,rg_2] - [rg_2,f_2] - [f_1,f_2]$ is a cycle of length $5$ in $L({\Gamma}(C(X))) $ containing $[f_1,f_2]$ and $[g_1,g_2]$. Hence,  $c([f_1,f_2],[g_1,g_2]) = 5$.\\
		
		(4). Suppose $c([f_1,f_2],[g_1,g_2]) = 6$. Then by (1), (2) and (3), $f_i \neq g_j$ for all $i,j \in \{1,2\}$ and $Z(f_i) \cap Z(g_j) = \phi$ for all $i,j \in \{1,2\}$.\\
		Conversely, suppose  $f_i \neq g_j$ for all $i,j \in \{1,2\}$ and $Z(f_i) \cap Z(g_j) = \phi$ for all $i,j \in \{1,2\}$, then by (1), (2) and (3),  $c([f_1,f_2],[g_1,g_2]) > 5$. By corollary 2.2, there exists a vertex $h \in  V({\Gamma}(C(X)))$ such that $h$ is adjacent to both $f_1$ and $g_1$ in ${\Gamma} (C(X))$. Consider, $r \in \mathbb{R} \setminus \{0\}$ then we get a cycle $[f_1,f_2] - [f_1,h] - [h,g_1] - [g_1,g_2] - [g_1,rh] - [rh,f_1] - [f_1,f_2]$ of length $6$ containing $[f_1,f_2]$ and $[g_1,g_2] $. Therefore, $c([f_1,f_2],[g_1,g_2]) = 6$. 
	\end{proof}
	
	\begin{theorem}
		The graph $L({\Gamma}(C(X))) $ is never chordal.
	\end{theorem}
	
	\begin{proof}
		Let $f \in V({\Gamma}C(X))$. Then $[f,3f] - [3f,5f] - [5f,7f] - [7f,f] - [f,3f]$ is a chordless cycle of length $4$ in $L({\Gamma}C(X)) $. Hence, $L({\Gamma}C(X)) $ is never chordal.
	\end{proof}

	\end{document}